\documentclass{amsart}
\usepackage{style}

\title[On the Galois-invariant part of the Weyl group]{On the Galois-invariant part of the Weyl group of the Picard lattice \\ of a K3 surface}
\author{Wim Nijgh \& Ronald van Luijk}
\date{\today}

\begin{document}

\begin{abstract}
    Let $X$ denote a K3 surface over an arbitrary field~$k$. Let $k^\text{s}$ denote a separable closure of~$k$ and let~$X^\text{s}$ denote the base change of~$X$ to~$k^\text{s}$. 
    The action of the absolute Galois group $\Gal(k^\text{s}/k)$ of~$k$ on~$\Pic X^\text{s}$ respects the intersection pairing, which gives $\Pic X^\text{s}$ the structure of a lattice.
    Let~$\tO(\Pic X)$ and~$\tO(\Pic X^\text{s})$ denote the group of isometries of~$\Pic X$ and~$\Pic X^\text{s}$, respectively. 
    Let~$R_X$ denote the Galois invariant part of the Weyl group of~$\tO(\Pic X^\text{s})$. One can show that each element in~$R_X$ can be restricted to an element of~$\tO(\Pic X)$. The following question arises: 
    \emph{Is the image of the restriction map $R_X \to \tO(\Pic X)$ a normal subgroup of~$\tO(\Pic X)$ for every K3 surface~$X$?} 
    We show that the answer is negative by giving counterexamples over~$k=\IQ$.
\end{abstract}

\maketitle

\section{Introduction}
In the last decades, K3 surfaces and their automorphisms have been widely studied. Recently, Bright e.a., \cite{BLL20}, generalized some finiteness theorems for K3 surfaces over algebraically closed fields of characteristic different from~2 to arbitrary fields, still of characteristic different from 2. One of the results they generalized is a well-known theorem about the automorphisms of K3 surfaces. To state this result and the main question of interest in this paper, we will start by introducing some definitions.

By a K3 surface over a field~$k$, we mean a smooth projective geometrically integral variety~$X$ over~$k$ of dimension~2 of which the canonical bundle is trivial and the first cohomology group $H^1(X,\cO_X)$ of the structure sheaf vanishes. 

For the remainder of this section, let $X$ denote a K3 surface over a field~$k$. It is known that the Picard group $\Pic X$ is a free abelian group of rank at most~22. We will denote the rank of $\Pic X$ by $\rho(X)$. The intersection pairing makes this group into a lattice as defined in~\autoref{s:genex}. We will denote this intersection pairing with $(\;.\;)$. The \emph{self-intersection number} of a class $D\in \Pic X$ is defined as~$D^2:=(D.D)$.

An isometry of $\Pic X$ is an automorphism of $\Pic X$ that respects the intersection pairing. We denote the group of isometries of $\Pic X$ by $\tO(\Pic X)$.

For every $(-2)$-class $C\in \Pic X$, i.e., a class with $C^2=2$, we define the reflection 
    $$\sigma_C\colon \Pic X\to \Pic X, \hspace{0.5cm} D \mapsto D+(C.D)C.$$

The Weyl group $\tW(\Pic X)$ is the group generated by all these reflections, i.e.,
\begin{equation*}
    \tW(\Pic X):=\langle \sigma_C \mid C\in \Pic X \text{ with } C^2=-2\rangle.
\end{equation*}

A direct verification shows that these reflections are isometries and one can show, which we will do in~\autoref{s:genex}, that $\tW(\Pic X)$ is a normal subgroup in $\tO(\Pic X)$.

In the algebraically closed case, whether or not the automorphism group of a K3 surface is finite, completely depends on the Picard lattice. This follows from the following theorem.

\begin{theorem}\label{thm:finresults} \textnormal{(\cite[Theorem 15.2.6]{Huy16}, \cite[Proposition 5.2]{LiM11})}
    Suppose $X$ is a K3 surface over an algebraically closed field of characteristic not equal to $2$.
    Then the natural map $$\Aut X\to \tO(\Pic X)/\tW(\Pic X)$$ has finite kernel and image of finite index.
\end{theorem}

In \cite{BLL20}, the result of \autoref{thm:finresults} is generalized to arbitrary fields. First of all, by~\cite[Lemma 3.1]{BLL20}, the statement of \autoref{thm:finresults} also holds for separably closed fields.
For arbitrary fields, a modification of the statement had to be made. 

Let~$k^\text{s}$ denote a separable closure of $k$ and let $\Gal(k^\text{s}/k)$ denote the absolute Galois group.
Denote by~$X^\text{s}$ the base change of~$X$ to~$k^\text{s}$. The group~$\Gal(k^\text{s}/k)$ acts on~$\Pic X^\text{s}$ through~$\tO(\Pic X^\text{s})$. 
Therefore, we get an induced action by conjugation of $\Gal(k^\text{s}/k)$ on the normal subgroup~$\tW(\Pic X^\text{s})$ of~$\tO(\Pic X^\text{s})$. We define the group~$R_X$ as the Galois invariant part of~$\tW(\Pic X^\text{s})$, i.e., 
$$R_X=\tW(\Pic X^\text{s})^{\Gal(k^\text{s}/k)}.$$

The group $R_X$ acts on the Galois invariant part $(\Pic X^\text{s})^{\Gal(k^\text{s}/k)}$. 
In \cite[Proposition~3.6(iii)]{BLL20}, it is proven that the action of $\tR_X$ on $\Pic X^\text{s}$ even preserves $\Pic X$. This means that we have a map $R_X\to \tO(\Pic X)$ given by restriction.
The question that has been raised by the authors of \cite{BLL20} is the following.
\begin{question}\label{question}
    \textit{Is the image of the restriction map $R_X\to \tO(\Pic X)$ a normal subgroup in $\tO(\Pic X)$ for each K3 surface $X$?} 
\end{question}
They note there is no reason to believe that the answer to this question is yes, although an example in which the image of $\tR_X$ is not normal is not found by the authors. 

To get an analogous statement of~\autoref{thm:finresults}, they observe that the natural map $\Aut X\to \tO(\Pic X^\text{s})$ induces an action of~$\Aut X$ on~$\tO(\Pic X^\text{s})$ by conjugation. Because $\tW(\Pic X^\text{s})$ is normal in $\tO(\Pic X^\text{s})$ and conjugation with an element of $\Aut X$ commutes with the Galois action, $\Aut X$ fixes the group~$R_X$. Hence, this gives an action of~$\Aut X$ on~$R_X$. 
The generalization of \autoref{thm:finresults} over arbitrary fields is as follows.
\begin{proposition}\textnormal{(\cite[Proposition 3.10]{BLL20})}
    Let $X$ be a K3 surface over a field~$k$ of characteristic different from~$2$. Then the natural map 
    \[\Aut X \ltimes \tR_X \to \tO(\Pic X)\]
    has finite kernel and image of finite index.
\end{proposition}
For more details and a proof of this proposition, we refer to the original paper.

In this paper, we show that the answer to~\autoref{question} is negative. We will do this in two steps. 
First, we state a more general question concerning only lattices, for which we give a counterexample. 
Then, we construct K3 surfaces for which the Picard lattice and the Galois invariant part of the Weyl group will coincide with this counterexample.

Computations for the example in \autoref{example} are done in \textsc{Magma} (\cite{magma}). The repository can be found at \href{
https://github.com/wimnijgh/K3-and-Weyl-group}{https://github.com/wimnijgh/K3-and-Weyl-group}.

\section{A non-normal subgroup in the group of isometries of a lattice}\label{s:genex}

A \emph{lattice} is a finitely generated free $\IZ$-module $\Lambda$ together with a symmetric non-degenerate bilinear form $b\colon \Lambda\times\Lambda\to \IZ$. We use the notation $x.y:=b(x,y)$ and $x^2:=b(x,x)$ for this bilinear form. If $x^2$ is even for all $x\in \Lambda$, we call the lattice \emph{even}.

If one chooses a basis $(x_1,\dots,x_n)$ for a lattice, the bilinear form can be represented by the Gram matrix 
\[\begin{pmatrix}
    x_1^2 & \cdots & x_1.x_n \\
    \vdots & \ddots & \vdots \\
    x_n.x_1 & \cdots & x_n^2
\end{pmatrix}.\]
The \emph{discriminant} of a lattice $\Lambda$, denoted $\disc \Lambda$, is given by the determinant of any Gram matrix representing the bilinear form. If $\Lambda'$ is a sublattice of $\Lambda$ of the same rank, then the relation
$$\disc \Lambda'=\disc \Lambda \cdot [\Lambda:\Lambda']^2,$$
holds for the discriminants, where $[\Lambda:\Lambda']$ denotes the index of~$\Lambda'$ in~$\Lambda$. This means that the discriminants of a lattice and a sublattice of full rank always differ by a square factor.

The group of \emph{isometries} of a lattice~$\Lambda$, denoted~$\tO(\Lambda)$, is the group of automorphisms of~$\Lambda$ that respect the bilinear form.

Next, we define the Weyl group of a lattice similarly to what we did in the introduction for the Weyl group of the Picard group. 
Our definition of the Weyl group will differ from some definitions in other literature. 
    Let $\Lambda$ be a lattice. Let $w\in \Lambda$ be an element such that $w^2=-2$. Then we define the \emph{reflection}
\begin{equation*}
    \sigma_w\colon \Lambda\to \Lambda, \hspace{0.5cm} x \mapsto x+(x.w)w.
\end{equation*}
We have that $\sigma_w$ is an isometry that satisfies $\sigma_w^2=\id$, hence the name reflection. 
The \textit{Weyl group} of $\Lambda$, denoted $\tW(\Lambda)$, is the group generated by all these reflections, i.e.,
\begin{equation*}
    \tW(\Lambda)=\langle \sigma_w \mid w\in \Lambda \text{ with } w^2=-2\rangle\subset \tO(\Lambda).
\end{equation*}
Observe that for $w\in \Lambda$ with $w^2=-2$, we have the relation
$\tau\circ\sigma_w=\sigma_{\tau(w)}\circ \tau$ 
for every $\tau\in \tO(\Lambda)$. It follows that the group $\tW(\Lambda)$ is normal in~$\tO(\Lambda)$. 

Next, let $G\subset \tO(\Lambda)$ be a subgroup such that $$\Lambda^G=\{x\in \Lambda \mid g(x)=x \text{ for all } g\in G\}$$ is a lattice, that is, the restriction of the bilinear form to $\Lambda^G$ is non-degenerate. 
The assumption that the restriction of the bilinear form to~$\Lambda^G$ is non-degenerate is a non-trivial condition as is shown in the next example.

\begin{example}
    Suppose that $\Lambda$ is a lattice of rank $\rho(\Lambda)=3$ for which a Gram matrix with respect to a basis $(x,y,z)$ can be given by
    \begin{align*}
        \begin{pmatrix}
            0 & 0 & 2 \\
            0 & -2 & 1 \\
            2 & 1 & 0
        \end{pmatrix}.
    \end{align*}
    Set $G:=\langle \sigma\rangle$, where $\sigma\in\tO(\Lambda)$ is the isometry given by $\sigma(x)= x$, $\sigma(y)=x+y$ and $\sigma(z)=y+z$. Then one easily checks that $\Lambda^G=\langle x\rangle$ and that the bilinear form is degenerate on $\Lambda^G$. 
\end{example}

Note that in this example, we have chosen $G$ to be infinite. The reason becomes clear due to the following result.

\begin{lemma}\label{lem:fin}
    Suppose that $\Lambda$ is a lattice and $G\subset \tO(\Lambda)$ is a finite group. Then the restriction of the bilinear form of $\Lambda$ to $\Lambda^G$ makes $\Lambda^G$ a lattice.
\end{lemma}
\begin{proof}
    To show that $\Lambda^G$ is a lattice, we show that the bilinear form of~$\Lambda$ restricted to~$\Lambda^G$ is non-degenerate by showing that for any $v\in \Lambda^G$ with~$v\neq 0$ there is an element $z\in \Lambda^G$ with $(v.z)\neq 0$.
    Let $v\in \Lambda^G$ be a non-zero element. Then there is an element $w\in \Lambda$ with $v.w\neq 0$. It follows that $z:=\sum_{g\in G} gw$ is~$G$-invariant and $v.z=\#G\cdot (v.w)\neq 0$. 
\end{proof}

Define the subgroup
\begin{equation*}
    \tO(\Lambda,\Lambda^G):=\{\sigma\in \tO(\Lambda) \mid \sigma(\Lambda^G)=\Lambda^G\}
\end{equation*}
of $\tO(\Lambda)$. Observe that there is a restriction map 
$$\res|_{\Lambda^G}\colon \tO(\Lambda,\Lambda^G)\to \tO(\Lambda^G), \hspace{0.5cm} \sigma\mapsto \sigma|_{\Lambda^G}.$$ 
Now let $N\subset \tO(\Lambda)$ be a normal subgroup of the group of isometries of $\Lambda$. Then the group~$G$ acts on~$N$ by conjugation and we denote by~$N^G$ the group of elements that are stable under this action, i.e., $$N^G=\{\sigma \in N\mid g\circ \sigma \circ g\iv=\sigma \text{ for all } g\in G\}.$$ 
For every $x\in \Lambda^G$, every $\sigma\in N^G$ and every $g\in G$, we have that $$\sigma(x)=(g\sigma g\iv)(x)=g(\sigma(x))$$ and so $\sigma(x)\in \Lambda^G$. It follows that $N^G$ is contained in the subgroup $\tO(\Lambda,\Lambda^G)$.

Recall that $\tW(\Lambda)$ is a normal subgroup of $\tO(\Lambda)$. By the above observations, we can look at the image of the group~$\tW(\Lambda)^G$ under the restriction map $\res|_{\Lambda^G}$. The image of the group~$\tW(\Lambda)^G$ is normal in the image of the restriction map, but in general, the image of~$\tW(\Lambda)^G$ does not have to be normal in $\tO(\Lambda^G)$ under this restriction map. We will give an example of this in \autoref{lem:notnormal}. 

\begin{lemma}\label{lem:notnormal}
    Let $\Lambda$ be a lattice of which the bilinear map is given by the Gram matrix
    $$\begin{pmatrix} 
        -4 & 3 & 3 \\
        3 & -2 & 0 \\
        3 & 0 & -2 
    \end{pmatrix},$$
    
    with respect to a basis $(x,y,z)$.
    Let $\tau\colon \Lambda\to \Lambda$ be the isometry switching~$y$ and~$z$, i.e., $\tau(x)=x$, $\tau(y)=z$ and $\tau(z)=y$. Set $G=\langle\tau\rangle$.
    Then $\Lambda^G$ is a lattice, which is generated by $(x,y+z)$. Moreover, the image of~$\tW(\Lambda)^G$ under the restriction map $\res|_{\Lambda^G}\colon \tO(\Lambda,\Lambda^G)\to \tO(\Lambda^G)$ is not normal in the group~$\tO(\Lambda^G)$.
\end{lemma}
\begin{proof}
    First observe that for $w:=ax+by+cz$, we have $\tau(w)=ax+cy+bz$. So~$\tau(w)=w$ if and only if $b=c$. It follows that the subgroup $\Lambda^G$ is generated by $x$ and~$d:=y+z$.

    With respect to the basis $(x,d)$, the Gram matrix of the bilinear map restricted to $\Lambda^G$ is given by
    $$\begin{pmatrix} 
        -4 & 6 \\
        6 & -4 
    \end{pmatrix}.$$    
    Observe that the bilinear map restricts to a non-degenerate bilinear map on~$\Lambda^G$ and hence~$\Lambda^G$ is a lattice, as can also be deduced from \autoref{lem:fin}.
    
    Now we define two isometries. First, we define $\sigma\in \tW(\Lambda)$ as $\sigma:=\sigma_{z}\circ \sigma_{y}$. 
    For~$v\in \Lambda$ we have that 
    $$\sigma(v)=\sigma_{z}(v+(v.y)y)=v+(v.y)y+((v+(v.y)y).z)z=v+(v.y)y+(v.z)z,$$ because $y.z=0$.
    It follows by symmetry that $\sigma=\sigma_{y}\circ \sigma_{z}$. 
    We also observe that~$\tau\circ\sigma_y\circ \tau\iv=\sigma_z$ and $\tau\circ \sigma_z\circ \tau\iv=\sigma_y$. From this we deduce that $$\tau\circ \sigma\circ \tau\iv=\tau \circ \sigma_y \circ \tau\iv \circ \tau \circ \sigma_z \circ \tau\iv= \sigma_z\circ \sigma_y=\sigma$$ and so $\sigma\in \tW(\Lambda)^G$. 
    
    Next, we define the isometry $\alpha\colon \Lambda^G \to \Lambda^G$ by $\alpha(x)=d$ and $\alpha(d)=x$. Observe that $\alpha^2=\id$ and so $\alpha\iv=\alpha$.   
    We now show that the image of $\tW(\Lambda)^G$ under the restriction map $\res|_{\Lambda^G}$ is not a normal subgroup in~$\tO(\Lambda^G)$, by showing that~${}^\alpha \sigma:=\alpha\circ (\sigma|_{\Lambda^G})\circ \alpha\iv\in \tO(\Lambda^G)$ cannot be extended to an element in~$\tO(\Lambda)$ and hence cannot be the restriction of an element in~$\tW(\Lambda)^G$. 

    We observe that
    \begin{align*}
        ({}^\alpha \sigma)(x)&=\alpha(\sigma(\alpha\iv(x)))=\alpha(\sigma(d))=\alpha(-d)=-x;\\
        ({}^\alpha \sigma)(d)&=\alpha(\sigma(\alpha\iv(d)))=\alpha(\sigma(x))=\alpha\left(x+(x.y)d\right)=d+3x.
    \end{align*}
    If ${}^\alpha \sigma$ extended to an isometry $\Tilde{\sigma}$ of $\Lambda$, then $\Tilde{\sigma}(y)=kx+ly+mz$ for some $k,l,m\in \IZ$, where these integers need to satisfy the following:
    \begin{align*}
        x.y=3& \; \text{ and so }\; 3=\Tilde{\sigma}(x).\Tilde{\sigma}(y)= 4k - 3 ( l + m); \\
        d.y=-2& \; \text{ and so }\; {-2}=\Tilde{\sigma}(d).\Tilde{\sigma}(y)= -6k + 7 ( l + m ).
    \end{align*}
    The first equation implies that $l+m$ is odd, whereas the second implies that it has to be even. Hence, ${}^\alpha \sigma$ cannot be extended to an isometry of $\Lambda^G$. From the above, we deduce that the image of $\tW(\Lambda)^G$ under the restriction map $\res|_{\Lambda^G}$ is not normal in $\tO(\Lambda^G)$.
\end{proof}

\begin{remark}
    There are only two ways to extend the isometry~${}^\alpha \sigma$ of the proof of \autoref{lem:notnormal} to isometries of the vector space $\Lambda\otimes \IQ$. Namely, the isometries~$\sigma_1,\sigma_2\colon \Lambda\otimes \IQ\to \Lambda\otimes \IQ$ defined by
    \begin{align*}
        \sigma_1(x):=-x,\; \sigma_1(y)&:=\tfrac{3}{2}x+y\; \text{ and }\; \sigma_1(z):=\tfrac{3}{2}x+z; \\
        \sigma_2(x):=-x,\;\sigma_2(y)&:=\tfrac{3}{2}x+z\; \text{ and }\; \sigma_2(z):=\tfrac{3}{2}x+y.
    \end{align*}
\end{remark}

The situation of \autoref{lem:notnormal} relates to our situation in the following way. Let~$X$ be a K3 surface over some field $k$. As in the introduction, we let~$X^\text{s}$ denote the base change of~$X$ to a separable closure $k^\text{s}$.
Set $\Lambda:=\Pic X^\text{s}$. The absolute Galois group acts on the group $\Lambda$ and this action factors via the group~$\tO(\Lambda)$. 
Choose~$G$ as the image of the map $\Gal(k^\text{s}/k)\to \tO(\Lambda)$. Then by definition of~$\tR_X$, we have $\tR_X=\tW(\Lambda)^G$. 
In general, we have that $\Pic X\subset \Lambda^G$ is a sublattice of finite index. If the equality $\Pic X=\Lambda^G$ holds, which is for example the case if $X(k)$ is non-empty, then the restriction of $\res|_{\Lambda^G}$ to~$R_X$ is exactly the restriction map $R_X\to \tO(\Pic X)$ from the question in the introduction.

The counterexample constructed in \autoref{lem:notnormal} shows that there is no reason to believe that the image of $\tR_X$ is normal in $\Pic X$. Even more is true. The result of \autoref{lem:notnormal} can be applied to a certain type of quartic surface in~$\IP^3_k$, that is, a smooth surface in~$\IP^3_k$ that is defined as the zero set of a polynomial of degree~$4$.

\begin{proposition}\label{prp:notnormal}
    Let~$X$ be a smooth quartic surface in~$\IP^3_k$.
    Assume that~$(H,L_0,L_1)$ is a basis for the Picard group $\Pic X^\text{s}$, where~$H$ denotes the hyperplane section on~$X$, and~$L_0$ and~$L_1$ are the classes of two skew lines on~$X^\text{s}$ that are conjugate over some quadratic extension of~$k$.
    Then~$\Pic X$ equals~$(\Pic X^\text{s})^{\Gal(k^\text{s}/k)}$ and the image of~$\tR_X$ in~$\tO(\Pic X)$ is not normal in~$\tO(\Pic X)$.
\end{proposition}
\begin{proof}
    By the above assumptions, with respect to the basis $(H,L_0,L_1)$, the intersection product on the Picard group $\Pic X^\text{s}$ is given by the matrix 
    $$\begin{pmatrix} 
        4 & 1 & 1 \\
        1 & -2 & 0 \\
        1 & 0 & -2 
    \end{pmatrix}.$$
    By choosing the basis $(D,L_0,L_1)$ with $D:=H-L_0-L_1$, we can deduce that this lattice is isomorphic to the lattice in \autoref{lem:notnormal} by the isomorphism 
    $$\Pic X^\text{s}\xrightarrow{\sim} \Lambda; \; D\mapsto x,\; L_0\mapsto y, \; L_1\mapsto z. $$
    Moreover, the image of the Galois group $\Gal(k^\text{s}/k)$ in~$\tO(\Pic X^\text{s})$ can be identified under this isomorphism with the group~$G$ of \autoref{lem:notnormal}. 

    By our assumption, $\Pic X$ contains the divisors $H$ and $L_0+L_1$, hence also~$D$. In combination with \autoref{lem:notnormal}, we find $$\langle D,L_0+L_1\rangle \subset\Pic X\subset (\Pic X^\text{s})^{\Gal(k^\text{s}/k)}=\langle D,L_0+L_1\rangle$$ and so the equality $\Pic X=(\Pic X^\text{s})^{\Gal(k^\text{s}/k)}$ follows. This means that we can identify $\Pic X$ with the lattice $\Lambda^G$ under the above isomorphism. It now follows as a corollary of \autoref{lem:notnormal} that the image of~$\tR_X$ in~$\tO(\Pic X)$ is not a normal subgroup.
\end{proof}

In the next section, we will construct explicit examples over $k=\IQ$ which satisfy the properties of the above proposition. But first, we prove one last fact about even lattices of odd rank, which gives us a useful result.
\begin{lemma}\label{lem:evendisc}
    Let $\Lambda$ be an even lattice. Suppose that the rank of $\Lambda$ is odd. Then the discriminant of~$\Lambda$ is even.
\end{lemma}
\begin{proof}
    Because the bilinear form on the lattice $\Lambda$ is even and symmetric, its base change to $\IF_2$ is an alternating bilinear form over $\IF_2$. Because every alternating bilinear form on a vector space of odd dimension is degenerate, this implies that $\disc \Lambda$ is even.
\end{proof}

\begin{corollary}\label{cor:eqL}
    Suppose that $\Lambda'$ is an even lattice of rank $3$ containing the lattice~$\Lambda$ of \autoref{lem:notnormal} as a sublattice. Then the lattice $\Lambda'$ equals $\Lambda$.
\end{corollary}
\begin{proof}
    First, we observe that the discriminant of $\Lambda$ equals $20=2^2\cdot 5$. Suppose that~$\Lambda$ is a sublattice of some lattice~$\Lambda'$ of rank $3$, then the discriminant of~$\Lambda'$ divides~$20$ and moreover, it differs by a square factor from~$20$. This implies that it is either~$5$ or~$20$. Because we assumed that $\Lambda'$ is even, it follows from \autoref{lem:evendisc} that it has to be $20$. This implies that $\Lambda$ equals $\Lambda'$.
\end{proof}

\section{Explicit examples over the rationals}\label{s:exQ}

We start this section by introducing some definitions and notation.
Let~$X$ be a scheme over $\Spec R$ where~$R$ denotes a commutative ring. Then for every ring homomorphism~$R\to S$, we denote by~$X_S$ the base change to~$S$, i.e., $X_S=X\times_{\Spec R} \Spec S$.

By a quartic in~$\IP^3_\IZ$ we mean a scheme of the form $\Proj \IZ[x,y,z,w]/(h)$, where $h\in\IZ[x,y,z,w]$ is a homogeneous polynomial of degree~4. If~$k$ is a field,~$X$ is a quartic in~$\IP^3_\IZ$, and the base change~$X_k$ is smooth, then~$X_k$ is a K3 surface.

Let~$Y$ be a K3 surface over a field~$k$ and let~$\overline{k}$ be an algebraic closure of~$k$. The \emph{geometric Picard number} of~$Y$ is defined as the Picard number of the base change~$Y_{\overline{k}}$, i.e., the number~$\rho(Y_{\overline{k}})$.

Our final objective is to construct explicit K3 surfaces over $\IQ$ with the properties of \autoref{prp:notnormal}. 
The following scheme will play an important role in this construction.
We define the scheme $C$ as
    \begin{align*}
        C:=\Proj&(\IZ[x,y,z,w]/(f_1,f_2,f_3))\subset \IP^3_\IZ, \; \text{ where}\\
        &f_1 := x^2-xy+6y^2;\\
        &f_2 := z^2-zw+6w^2;\\
        &f_3 := yz-xw.
    \end{align*}
Note that the roots of the polynomial $t^2-t+6$ in $\overline{\IQ}$ are given by $\beta_0:=\frac{1+\sqrt{-23}}{2}$ and $\beta_1:=\frac{1-\sqrt{-23}}{2}=1-\beta_0$. Define the quadratic number ring $R:=\IZ[\beta_0]$. Then the base change~$C_R$ is given by the union of the two schemes $L_0$ and $L_1$, where
\begin{align*}
    L_i:= \Proj(R[x,y,z,w]/(x-\beta_i y, z-\beta_i w)).
\end{align*}
Moreover, the base change of $C$ to the fraction field $Q(R)\subset \overline{\IQ}$ of~$R$ gives two skew lines that are conjugate over~$\IQ$.


\begin{lemma}\label{lem:qcC}
    Let $X$ be a quartic in $\IP_\IZ^3$ that contains $C$. If $X_\IQ$ is smooth and has geometric Picard number 3, then $\Pic X$ equals $(\Pic X^\text{s})^{\Gal(k^\text{s}/k)}$ and the image of~$\tR_X$ in~$\tO(\Pic X)$ is not normal in~$\tO(\Pic X)$.
\end{lemma}
\begin{proof}
    Note that we have $\langle H,L_{0,\overline{\IQ}},L_{1,\overline{\IQ}}\rangle\subset \Pic X_{\overline{\IQ}}$, where~$H$ denotes the hyperplane section. By the assumption that the geometric Picard number equals~$3$, the fact that~$H,L_{0,\overline{\IQ}},L_{1,\overline{\IQ}}$ are linearly independent, and \autoref{cor:eqL}, it follows that $\Pic X_{\overline{\IQ}}=\langle H,L_{0,\overline{\IQ}},L_{1,\overline{\IQ}}\rangle$. The result now follows from~\autoref{prp:notnormal}.
\end{proof}

Now our question reduces to finding a quartic $X$ satisfying the properties of \autoref{lem:qcC}. We expect that if we pick a general quartic containing~$C$, that~$X_\IQ$ has geometric Picard number~$3$. But showing that this equality holds for a concrete example, requires a bit more work. Therefore, we will follow the method described by one of the authors in \cite[Theorem 5.2.2]{vL05}. This entails in our situation the following.

\begin{proposition}\label{prp:redquartic}
    Let $X$ be a quartic in $\IP_\IZ^3$ that contains $C$. Suppose that  the reductions $X_{\IF_2}$ and $X_{\IF_3}$ are both smooth. Suppose furthermore that
    \begin{itemize}
        \item[(i)] the reduction $X_{\IF_2}$ contains the line $L_2$ given by $y=w=0$;
        \item[(ii)] the reduction $X_{\IF_3}$ contains the line $L_3$ given by $x=y=0$;
        \item[(iii)] the geometric Picard numbers of $X_{\IF_2}$ and $X_{\IF_3}$ are both equal to $4$.
    \end{itemize}
    Then $X_\IQ$ is smooth and has geometric Picard number $3$. 
\end{proposition}
\begin{proof}
    Take $p\in\{2,3\}$. The fact that~$X_\IQ$ is smooth follows directly from the fact that~$X_{\IF_p}$ is smooth, so we are left to prove that~$\rho(X_{\overline{\IQ}})$ equals~$3$.

    Let~$H_p$ denote the hyperplane section on~$X_{\IF_p}$. The curve~$C_{\IF_p}$ is the union of two skew lines. If we choose a prime $\mathfrak{p}\subset R$ above $p$, then $R/\mathfrak{p}=\IF_p$ and the skew lines are exactly the base changes $L_{0,p}$ and $L_{1,p}$ with $L_{i,p}:=(L_i)_{\IF_p}$.
    Moreover, we may choose $\mathfrak{p}$ such that~$\beta_0$ reduces to~$0$ and $\beta_1$ reduces to~$1$. Then~$L_{0,p}$ is the intersection of the hyperplanes given by the relation $x=z=0$, and the other line $L_{1,p}$ is given by the equations $x=y$ and $z=w$.

    First, we take a closer look at the intersection form on $X_{\IF_2}$. We have that the intersection form on the subgroup generated by $H_2,L_{0,2},L_{1,2},L_2$ is given by the Gram matrix
    \begin{align*}
    \begin{pmatrix}
        4 & 1 & 1 & 1 \\
        1 & -2 & 0 & 0 \\
        1 & 0 & -2 & 0 \\
        1 & 0 & 0 & -2
    \end{pmatrix}.
    \end{align*}
    Note that the determinant of this matrix equals $44=2^2\cdot 11$, and by the assumption that  $\rho(X_{\overline{\IF}_2})=4$, we deduce that $H_2,L_{0,2},L_{1,2},L_2$ generate a sublattice of full rank of the Picard lattice of $X_{\overline{\IF}_2}$. It follows that the discriminant of $\Pic X_{\overline{\IF}_2}$ equals $11$ or $44$.

    For $X_{\IF_3}$, we deduce similarly that $H_3,L_{0,3},L_{1,3},L_3$ generate a sublattice of full rank of the Picard lattice of $X_{\overline{\IF}_3}$, where the intersection form is given by the Gram matrix
    \begin{align*}
    \begin{pmatrix}
        4 & 1 & 1 & 1 \\
        1 & -2 & 0 & 1 \\
        1 & 0 & -2 & 1 \\
        1 & 1 & 1 & -2
    \end{pmatrix}.
    \end{align*}
    The determinant of this matrix equals $36=2^2\cdot 9$, which means that the discriminant of $\Pic X_{\overline{\IF}_3}$ equals $9$ or $36$.
    It follows that the discriminants of~$\Pic X_{\overline{\IF}_2}$ and $\Pic X_{\overline{\IF}_3}$ do not differ by a rational square factor.

    By reduction, we have inclusions $\Pic X_{\overline{\IQ}} \hookrightarrow \Pic X_{\overline{\IF}_p}$ for both $p=2,3$, see \cite[Proposition~2.6.2]{vL05}. 
    So the lattice $\Pic X_{\overline{\IQ}}$ can be identified as a sublattice of both. 
    Now $\Pic X_{\overline{\IQ}}$ can only be a sublattice of rank~4 if the discriminant of this lattice $\Pic X_{\overline{\IQ}}$ differs by a square factor from the discriminant of both lattices $\Pic X_{\overline{\IF}_2}$ and $\Pic X_{\overline{\IF}_3}$. 
    But because these two lattices do not differ by a square factor, it follows that $\rho(X_{\overline{\IQ}})<4$.
    Because we know that $X_{\overline{\IQ}}$ contains the lines~$L_{0,\overline{\IQ}}$ and~$L_{1,\overline{\IQ}}$ and the hyperplane section, the geometric Picard number is at least~$3$, and hence we have the equality $\rho(X_{\overline{\IQ}})=3$. 
\end{proof}

To construct an example that satisfies the properties of the above proposition, we searched for smooth surfaces over $\IF_2$ and $\IF_3$ respectively with the given properties. Using \textsc{Magma}, we found the examples described in~\autoref{example}. For details of this search, we refer to the
\href{https://github.com/wimnijgh/K3-and-Weyl-group}{GitHub repository} of this paper. 

\begin{theorem}\label{example}
    Let $X\subset \IP^3_\IZ$ be the quartic given by the zero set of the polynomial~$h\in \IZ[x,y,z,w]$ with $h:=q_1f_1+q_2f_2+q_3f_3$, where $q_1,q_2,q_3\in\IZ[x,y,z,w]$ are homogeneous polynomials of degree~2 such that
    \begin{align*}
        q_1 &\equiv y^2 + xw + xy + yw \pmod{2}, \\
        q_2 &\equiv xy + yz + yw + zw \pmod{2}, \\
        q_3 &\equiv xz + xw + z^2 + w^2 + y^2 \pmod{2},
    \end{align*}
    and
    \begin{align*}
        q_1 &\equiv y^2 + xw + 2xz + 2yz \pmod{3}, \\
        q_2 &\equiv xy + yz + yw + x^2 + 2 xw \pmod{3}, \\
        q_3 &\equiv xz + xw + z^2 + w^2 + 2xy + 2yw \pmod{3}.
    \end{align*}
    Then $X_\IQ$ is a K3 surface with $\Pic X_\IQ=(\Pic X_{\overline{\IQ}})^{\text{Gal}(\overline{\IQ}/\IQ)}$ and the image of~$\tR_{X_\IQ}$ in~$\tO(\Pic X_\IQ)$ is not normal.
\end{theorem}

\begin{proof}
    We will show that $X$ satisfies the properties of \autoref{prp:redquartic}. Then the result will follow from \autoref{lem:qcC}.

    From the definition of $h$, it is clear that the surface~$X$ contains the curve~$C$. Furthermore, a direct verification shows that the reductions~$X_{\IF_2}$ and~$X_{\IF_3}$ are smooth.
    The reduction of the polynomials~$q_1,q_2,f_3$ to~$\IF_2[x,y,z,w]$ are contained in the ideal~$(y,w)$, hence~$X_{\IF_2}$ contains the line~$L_2$ described as in \autoref{prp:redquartic}. Similarly, the reduction of the polynomials~$f_1,q_2,f_3$ to~$\IF_3[x,y,z,w]$ are contained in the ideal~$(x,y)$ and so~$X_{\IF_3}$ contains the line~$L_3$.
    Hence, we are left to prove that the geometric Picard numbers of~$X_{\IF_2}$ and~$X_{\IF_3}$ are both~4. We follow the method used in the proof of \cite[Theorem 5.2.2]{vL05}.
    
    Take $p\in\{2,3\}$ and denote by $\Phi_p$ the absolute Frobenius on~$X_{\IF_p}$, i.e., it is the identity map on~$X_{\IF_p}$ and on the structure sheaf it sends an element~$g$ to~$g^p$. Let~$\overline{\Phi}_p^*$ denote the automorphism on the second \'etale cohomology group~$H^2_\text{\'et}(X_{\overline{\IF}_p},\IQ_\ell(1))$ induced by the base change $\overline{\Phi}_p\colon X_{\overline{\IF}_p}\to X_{\overline{\IF}_p}$.

    The Picard group of $X_{\overline{\IF}_p}$ embeds in the space $H^2_\text{\'et}(X_{\overline{\IF}_p},\IQ_\ell(1))$ and this embedding is Galois-invariant, see \cite[Proposition 2.6.2]{vL05}.
    Because all classes in the geometric Picard group can already be defined over some finite extension of~$\IF_p$ and~$\Pic X_{\IF_p}$ is finitely generated, there is some power of Frobenius which acts trivially on this group, and hence also on the whole image of~$\Pic X_{\overline{\IF}_p}$ in~$H^2_\text{\'et}(X_{\overline{\IF}_p},\IQ_\ell(1))$. 
    This shows that the number of eigenvalues of~$\overline{\Phi}_p^*$ which are roots of unity, counted with multiplicity, is an upper bound of the geometric Picard number $\rho(X_{\overline{\IF}_p})$ of~$X_{\IF_p}$. 
    (In fact, this upper bound equals the geometric Picard number of~$X_3$, because this is implied by the Tate conjecture which is proven for K3 surfaces, see \cite{Pera15} and \cite{Pera16}.)

    To find the eigenvalues, we want to calculate the characteristic polynomial of~$\overline{\Phi}_p^*$, which we will denote by~$f_p$, and continue as follows. Let $V$ be the quotient space $V:=H^2_\text{\'et}(X_{\overline{\IF}_p},\IQ_\ell(1))/U$, where $U$ denotes the subspace generated by the images of the classes of~$H$, $(L_0)_{\IF_p}$, $(L_1)_{\IF_p}$ and $L_p$. Because $\overline{\Phi}_p^*$ leaves all elements of~$U$ invariant, we get an induced automorphism $\varphi_p\colon V\to V$.
    Moreover, we have the relation 
    \begin{equation*}
        f_p(t)=(t-1)^4\cdot f_{\varphi_p}(t).
    \end{equation*}
    where~$f_{\varphi_p}$ denotes the characteristic polynomial of $\varphi_p$. 
    
    Note that $\dim V=\dim H^2_\text{\'et}(X_{\overline{\IF}_p},\IQ_\ell(1))-\dim U=22-4=18$. 
    So the characteristic polynomial of $\varphi_p$ is a degree 18 polynomial, which is equal to
    $$f_{\varphi_p}(t)=t^{18}+c_1t^{17}+\dots + c_{18},$$
    where $c_1=-\Tr(\varphi_p)$ and the other~$c_i$ are given recursively by 
        \begin{equation*}
            c_i=-\frac{\Tr(\varphi_p^i)+\sum_{j=1}^{i-1}c_j\Tr(\varphi_p^{i-j})}{i},
        \end{equation*}
    see \cite[Lemma 5.2.1]{vL05}.
    Using the Lefschetz trace formula and the Weil conjectures, together with the fact that $\Tr(\varphi_p^n)=\Tr({\overline{\Phi}_p^*}^n)-4$, we can deduce the equality
    \begin{equation*}\label{eq:trace}
        \Tr(\varphi_p^n)=\frac{\#X_{\IF_p}(\IF_{p^n})-1-p^{2n}-4p^n}{p^n}.
    \end{equation*}
    
    Using \textsc{Magma}, we count the points on $X_{\IF_p}$ over the field extensions~$\IF_{p^n}$ for $1\leq n\leq 9$. The outcome is listed in the table below.
    \begin{center}\Small
    \begin{tabular}{r|lllllllll}
        $n$              & 1 & 2 & 3 & 4 & 5 & 6 & 7 & 8 & 9 \\ \hline
        $\#X_2(\IF_{2^n})$ & 9 & 33 & 105 & 257 & 1249 & 4161 & 17089 & 65537 & 264705 \\
        $\#X_3(\IF_{3^n})$ & 21 & 95 & 819 & 6983 & 59406 & 534179 & 4804401 & 43025999 & 387536184
    \end{tabular}
    \end{center}
    This gives us the following first nine coefficients of $f_{\varphi_p}$.
    \begin{center}\small 
    \begin{tabular}{r|lllllllll}
        & $c_1$ & $c_2$ & $c_3$ & $c_4$ & $c_5$ & $c_6$ & $c_7$ & $c_8$ & $c_9$ \\ \hline
        $p=2$ & $2$ & $2$ & $1$ & $1$ & $1$ & $1$ & $\tfrac{1}{2}$ & $1$ & $1$ \\
        $p=3$ & $\tfrac{1}{3}$ & $\tfrac{4}{3}$ & $\tfrac{2}{3}$ & $\tfrac{2}{3}$ & $1$ & $\tfrac{1}{3}$ & $0$ & $1$ & $-\tfrac{2}{3}$
    \end{tabular}
    \end{center}
    The Weil conjectures give the functional equation $f_p(t)=\pm t^{22}f_p(1/t)$. On both sides we can divide by $(t-1)^4$ giving the relation $f_{\varphi_p}(t)=\pm t^{18} f_{\varphi_p}(1/t)$.
    For both $p$, we have that $c_9\neq 0$, and hence it follows that the sign of the functional equation should be positive. This gives the relation $c_{18-i}=c_i$ for all $1\leq i\leq 9$. Because~$f_p$ is monic, it follows that $c_{18}=1$.
    We deduce that
    \begin{align*}
     \begin{split}
        f_2(t)=(t-1)^4 (t^{18} + &2t^{17} + 2t^{16} + t^{15} + t^{14} + t^{13} + t^{12} + \tfrac{1}{2}t^{11} + t^{10} + \\
         & t^9 + t^8 + \tfrac{1}{2}t^7 + t^6 + t^5 + t^4 + t^3 + 
        2t^2 + 2t + 1);
     \end{split} \\
     \begin{split}
        f_3(t)=(t-1)^4(t^{18} + &\tfrac{1}{3}t^{17} + \tfrac{4}{3}t^{16} + \tfrac{2}{3}t^{15} + \tfrac{2}{3}t^{14} + t^{13} + \tfrac{1}{3}t^{12} + t^{10} - \\
        & \tfrac{2}{3}t^9 + t^8 + \tfrac{1}{3}t^6 + t^5 + \tfrac{2}{3}t^4 + \tfrac{2}{3}t^3 + \tfrac{4}{3}t^2 + \tfrac{1}{3}t + 1).
     \end{split}
    \end{align*}
    
    Now let $\phi$ denote the Euler-phi-function, i.e.,~$\phi(n)=\#(\IZ/n\IZ)^*$. One can verify that if~$\phi(n)\leq 18$, we have that~$n\leq 60$. For every~$n^\text{th}$-root of unity~$\zeta$ with~$1\leq n\leq 60$, we can check that $f_{\varphi_p}(\zeta)\neq 0$ and so it is not a zero of $f_{\varphi_p}$. Hence, we deduce that~$f_p$ has no other eigenvalues which are a root of unity, except for the eigenvalue~$1$ which has multiplicity~$4$. It follows that~$4$ is an upper bound for the geometric Picard number of the surface~$X_{\IF_p}$. 

    Recall that~$X$ is constructed in such a way that~$X_{\IF_p}$ has at least four linearly independent divisors. 
    It follows that $\rho(X_{\overline{\IF}_p})=4$ and hence we deduce that~$X_\IQ$ satisfies the properties of \autoref{lem:qcC}. 
    We conclude that~$\Pic X_\IQ$ equals $(\Pic X_{\overline{\IQ}})^{\text{Gal}(\overline{\IQ}/\IQ)}$ and that the image of~$\tR_{X_\IQ}$ in~$\tO(\Pic X_\IQ)$ is not normal.
\end{proof}

\begin{corollary}
    The answer to \autoref{question} is no.
\end{corollary}
\begin{proof}
    One can construct counterexamples by using the Chinese remainder theorem on the polynomials given in \autoref{example}.
\end{proof}

\bibliographystyle{amsplain}
\bibliography{references}

\end{document}